\documentclass[preprint,12pt]{elsarticle}

\usepackage{amsfonts,amsthm,amsmath}
\usepackage{mathrsfs}
\DeclareMathOperator{\Ff}{Fact}
\DeclareMathOperator{\Pre}{Pref}

\DeclareMathOperator{\card}{card}
\DeclareMathOperator{\orb}{Orb}
\DeclareMathOperator{\alf}{alph}

\newcommand{\Nn}{\mathbb N}

\newcommand{\Sh}{{\rm Sh}}

\newcommand{\nats}{{\mathbb N}}

\newtheorem{thm}{Theorem}
\newtheorem{question}[thm]{Question}
\newtheorem{theorem}{Theorem}[section]
\newtheorem{lemma}[theorem]{Lemma}
\newtheorem{corollary}[theorem]{Corollary}
\newtheorem{proposition}[theorem]{Proposition}
\theoremstyle{definition}
\newtheorem{remark}[theorem]{Remark}
\newtheorem{example}[theorem]{Example}
\newtheorem{definition}[theorem]{Definition}

\begin{document}

\begin{frontmatter}

\title{A Coloring Problem for  Infinite Words}

\author[label1]{Aldo de Luca\fnref{label5}\corref{cor1}}
  \ead{aldo.deluca@unina.it}

    \fntext[label5]{Partially supported by GNSAGA of CNR.}

  \author[label2]{Elena V. Pribavkina\fnref{label6}}
  \ead{elena.pribavkina@usu.ru}

  \fntext[label6]{Supported under the Agreement 02.A03.21.0006 of 27.08.2013
between the Ministry of Education and Science of the Russian
Federation and Ural Federal University;
supported by the Presidential Programm for young researchers, grant MK-3160.2014.1 and
by the Russian Foundation for Basic research, grant 13-01-00852.}

   \author[label3,label4]{Luca Q. Zamboni\fnref{label7}}
  \ead{lupastis@gmail.com}

  \fntext[label7]{Partially supported by a FiDiPro grant (137991) from the Academy of Finland and by
ANR grant {\sl SUBTILE}.}
 \cortext[cor1]{Corresponding author.}
\address[label1]{Dipartimento di Matematica e Applicazioni,
Universit\`a di Napoli Federico II, Italy}
\address[label2]{Institute of Mathematics and Computer Science, Ural Federal University, Ekaterinburg, Russia}
\address[label3]{FUNDIM, University of Turku, Finland}
\address[label4]{Universit\'e de Lyon,
Universit\'e Lyon 1, CNRS UMR 5208,
Institut Camille Jordan,
43 boulevard du 11 novembre 1918,
F69622 Villeurbanne Cedex, France}

\begin{abstract}
In this paper we consider the following question in the spirit of Ramsey theory: Given $x\in A^\omega,$ where $A$ is a finite
non-empty set, does there exist a finite coloring
 of the non-empty factors of $x$
with the property that no factorization of  $x$ is monochromatic? We prove that this question has a positive answer using two colors  for  almost all words relative to the standard Bernoulli measure on $A^\omega.$ We also show that it has a positive answer for various classes of uniformly recurrent words, including all aperiodic balanced words, and all words $x\in A^\omega$ satisfying
$\lambda_x(n+1)-\lambda_x(n)=1$ for all $n$ sufficiently large, where $ \lambda_x(n)$ denotes the number of distinct factors of $x$ of length $n.$
\end{abstract}

\journal{Journal of Combinatorial Theory Series A}

\end{frontmatter}

\section{Introduction}

In this paper we address the following question posed  by T. C. Brown  \cite{BTC} and independently by the third author  \cite{LQZ}:

\begin{question}\label{que:prima}
Given $x\in A^\omega,$ where $A$ is a finite non-empty set, does there exist a finite coloring
$ c:  \Ff ^+ x  \rightarrow C $ of the set of non-empty factors of $x$
with the property that for each factorization $x= U_1U_2U_3 \cdots$, there exist positive integers
$i,j$ for which $c(U_i)\neq c(U_j)$ ?
\end{question}

We shall refer to Question~\ref{que:prima} as  the {\em coloring problem}  for infinite words. We let ${\cal P}$ denote
the collection of all infinite words over any  finite alphabet for which the coloring problem has a positive answer, and  ${\cal P}_k$ (with $k\geq 2)$ the collection of all $x\in {\cal P}$ for which there exists a coloring map $ c:  \Ff ^+ x  \rightarrow C $ with $\card (C)=k$ such that for each factorization $x= U_1U_2U_3 \cdots$, there exist positive integers $i,j$ with $c(U_i)\neq c(U_j).$ A  related, and perhaps even more difficult question  is the following

\begin{question}\label{sue:seconda} If  an infinite word $w$ is in ${\cal P}$, then what  is the least integer $k$ such that
$w$ is in  ${\cal P}_k$ ?
\end{question}

It is evident that if $x\in A^\omega$ is purely periodic, i.e., $x=u^\omega$ for some $u\in A^+,$ then for each coloring $ c:  \Ff ^+ x  \rightarrow C ,$ the factorization $x=uuu\cdots$ is monochromatic; hence no purely  periodic word belongs to $ {\cal P}.$
Henceforth, for simplicity,  we use the term periodic to mean purely periodic and the term ultimately periodic for a word $x= vu^{\omega}$ for some $v\in A^*$, so non-periodic means not purely periodic. The term aperiodic will mean not ultimately periodic.   We know of no example of a non-periodic word which does not belong to ${\cal P}$ and conjecture that ${\cal P}$ contains all non-periodic words.

Consider the binary word $x=ab^\omega \in \{a,b\}^\omega.$ Then $x\in {\cal P}_2.$ In fact it suffices to color a factor $u$ of $x$ by
$$ c(u) = \begin{cases}
a, & \text{if $u \in ab^*$}; \\ b, &  \text{if $u \in b^+$}.
\end{cases} $$
Clearly,  relative to this $2$-coloring $c$ of $\Ff^+ x$,  no factorization of $x$ is monochromatic. However,  we note that the suffix $b^\omega$ does admit a monochromatic factorization with respect to $c.$ This phenomenon is in general unavoidable. In fact,  given any $x\in A^\omega$ and any finite coloring  $ c:  \Ff ^+ x  \rightarrow C ,$ there exists a suffix $x'$ of $x$ which admits a monochromatic factorization relative to $c.$ This fact may be obtained via a straightforward application of the infinite Ramsey theorem  (see Proposition~\ref{Ram} or  \cite{schutz} for a proof by M. P. Sch\"utzenberger which does not use Ramsey's theorem).
This suggests that given a non-periodic word $x\in A^\omega,$ a finite coloring $c$ which avoids a monochromatic factorization of $x$ must depend in part on the word $x$ and not just  on its set of factors.

We show that relative to the standard Bernoulli measure on $A^\omega$,  almost all words are in ${\cal P}_2$. More precisely, we consider words of  full  complexity, i.e., for every $n\geq 0$, the number $\lambda_x(n)$ of distinct
factors of $x$ of length $n$ reaches its maximal value namely $d^n$, with $d= \card(A)$.
 We prove that full complexity  words do not admit prefixal factorizations. Hence, given a full complexity word $s$, if we color the factors of $s$ according to the rule $c(u)=0$ if  $u$  is a prefix of $s$, and $c(u)=1$ otherwise, it follows that relative to this coloring rule, $s$ does not admit a monochromatic factorization. It is well known \cite{AS} that relative to the Bernoulli measure, almost all words are of full complexity.

 We next consider the class of all non-uniformly recurrent words, which properly contains the class of full complexity words.
We prove that every non-uniformly word belongs to ${\cal P}_k$  for some $k.$ Our proof uses the notion of derived words in the sense of F. Durand \cite{Du} and  some
 invariance properties of the coloring problem.
Essentially, we prove that each non-uniformly recurrent word $s$ has a derived word  which does not admit a prefixal factorization, and hence which is in ${\cal P}_2$. From this we deduce that  $s$ belongs to  ${\cal P}_k$ for some $k$ which
depends on the derived word.
This allows us to restrict Question~\ref{que:prima} to the class of all non-periodic uniformly recurrent words.

Uniformly recurrent words typically possess a rigid combinatorial structure, and examples suggest that given a non-periodic uniformly recurrent word $x\in A^\omega,$ any coloring $c$ which avoids a monochromatic factorization of $x$ typically reflects one or more combinatorial properties of $x.$
In this paper, we show that the coloring problem has a positive answer for various families of non-periodic uniformly recurrent words.

The most basic non-periodic uniformly recurrent words are Sturmian words. Their origins can be traced back to the astronomer J. Bernoulli  III in 1772 \cite{BJ}.  Sturmian words are infinite words having exactly $n+1$ factors of length $n$ for each $n \geq 0.$  They can be  regarded to be those non-periodic words which are closest to periodic words.
Sturmian words arise naturally in different areas of mathematics including
combinatorics, algebra, number theory, ergodic theory, dynamical systems, and differential equations.
They are also of great importance in theoretical physics (as basic  examples of $1$-dimensional quasicrystals \cite{AM}), and in theoretical computer science, where they are used in
computer graphics as digital approximations of straight lines (see, for instance, \cite{KR}).

In this paper we prove that the  coloring problem has a positive answer for all Sturmian words,
thereby solving a problem raised in both \cite{BTC} and \cite{LQZ}.  More precisely, we show that given any Sturmian word $w,$ there exists a $3$-coloring of its set of factors relative to which $w$ does not admit a monochromatic factorization. Our  proof is constructive and relies on  some  combinatorial properties of Sturmian words which may be of independent interest.
As a consequence of the Sturmian case,  we prove that given positive integers $N$ and $k,$ then any  word $x$  having exactly $n+k$ distinct factors of each length $n\geq N$ is in ${\cal P}.$
As another consequence we show that all aperiodic balanced words belong to ${\cal P}$.  An infinite word  $x$ over the alphabet  $A$  is balanced if for every letter $a\in A$ the numbers of occurrences of $a$ in any two of  factors $x$ of equal length differ by at most $1$. An aperiodic balanced word over  a two-letter alphabet is a Sturmian word (see Section~\ref{Sec:sw}).

The paper is organized as follows: In Section~\ref{bpcp} we investigate  basic properties of the class ${\cal P}$ mainly related to the prefixal factorizations of infinite words. In particular,
we prove that almost all  words are in ${\cal P}_2$. Moreover, we show that the coloring problem has a positive answer for various classes of non-periodic words,  including all overlap-free words (and hence all square-free words), all Lyndon words, and all standard episturmian words. In
 Section~\ref{sec:ipp} we prove some basic 
 invariance properties of the coloring problem.
 In Section~\ref{sec:nur} we show that ${\cal P}$ contains all non-uniformly recurrent words.
Section~\ref{Sec:sw} is devoted to the proof that every Sturmian word is in ${\cal P}$ and to its consequences.

An extended abstract  of some of the results of this paper appeared  in the Proceedings of DLT 2013
  \cite{adleplz}.

\section{Preliminaries}\label{Prel}

Given a  non-empty set $A,$ or {\em alphabet},  we let  $A^*$ denote the set of
all finite words $u=u_1u_2\cdots u_n$ with $u_i\in A.$ The
quantity $n$ is called the {\em length} of $u$ and is denoted $|u|.$
The {\em empty word}, denoted $\varepsilon,$ is the
unique element in $A^*$ with $|\varepsilon|=0.$
We set
$A^+=A-\{\varepsilon\}.$ For each word $v\in A^+$, let $|u|_v$  denote the number of occurrences
of $v$ in $u$. In the following we suppose that the alphabet $A$ is finite even though several results hold true
for any alphabet.

Given words $u,v\in A^+$ we say $v$ is a {\it border} of $u$ if $v$ is both a proper prefix and a proper suffix of $u.$ In case $u$ admits a border, we say $u$ is {\it bordered}. Otherwise $u$ is called {\it unbordered}. Two words $u$ and $v$ are said to be {\it conjugate} if there exist $r$ and $t$ such that $u=rt$ and $v=tr.$  A word $u\in A^+$ is called {\it Lyndon} if there exists a linear order on $A$ with respect to which $u$ is lexicographically smaller than all its conjugates. For instance, $ababb$ is Lyndon with respect to the order $a<b.$ It is readily checked that every Lyndon word is unbordered.

  Let  $A^\omega$  denote the set of all
one-sided infinite words $s=s_1s_2\cdots $ with $s_i\in A.$ We endow $A^\omega$ with the topology generated by the metric
\[d(s, t)=\frac 1{2^n}\,\, \mbox{where} \,\, n=\inf\{k \mid s_k\neq t_k\}\] 
whenever $s=(s_n)_{n\geq 1}$ and $t=(t_n)_{n\geq 1}$ are two elements of $A^\omega.$
The resulting topology is generated by the collection of {\it cylinders} $ [a_1,\ldots ,a_n]$, where
for each $n>0$ and $a_i\in A$, $1\leq i \leq n$,
\[
[a_1,\ldots ,a_n]=\{s\in A^\omega \,|\, s_i=a_i\, \ \mbox{for}\, 1\leq i\leq n\}.\]
It can also be described as being the product topology on $A^\omega$ with the discrete topology on  $A$. In particular this topology is compact. Each probability vector ${\mathbf p}=(p_i)_{i\in A}$ determines a measure $\mu_{\mathbf p}$ ({\it Bernoulli measure} corresponding to ${\mathbf p})$ defined as the unique measure on the $\sigma$-algebra of $A^\omega$ such that $\mu_{\mathbf p}([a_1,\ldots ,a_n])=p_{a_1}\cdots p_{a_n}.$ By the {\it standard Bernoulli measure} on $A^\omega$ we mean the Bernoulli measure corresponding to the probability vector ${\mathbf p}=(\frac 1{d}, \ldots , \frac 1{d})$, with $d=\card(A)$.

Given $s\in A^\omega,$ let
$\Ff^+ s=\{s_is_{i+1}\cdots s_{i+j}\,|\, i\geq 1,j\geq 0\}$ denote the set of all non-empty {\it factors} of $s$. Moreover, we set
$\Ff s =  \{\varepsilon\} \cup \Ff^+ s$.
The {\em factor complexity} of $s$ is the map $\lambda_s: \Ff s \rightarrow \Nn$ defined as follows: for any $n\geq 0$
$$ \lambda_s(n) = \card( A^n \cap \Ff s),$$
i.e., $\lambda_s(n)$ counts the number of distinct factors of $s$ of length $n$.
For any  finite or infinite word $s$ over the alphabet $A$, let  $\alf s$ denote  the set of all letters of $A$ occurring in $s$.
 A  factor $u$ of  $s$  is called {\em right special} (resp., {\em left special}) if there exist two different letters $x$ and $y$  such that $ux$ and $uy$ (resp., $xu$ and $yu$) are factors of $s$.

Given $s=s_1s_2s_3\cdots \in A^\omega,$ and $u\in \Ff^+s,$ let $s\big|_u$ denote the set of all occurrences of $u$ in $s,$ i.e., \[s\big|_u=\{n\,|\, s_ns_{n+1}\cdots s_{n+|u|-1}=u\}.\]
A factor $u$ of  $s \in A^\omega$ is called {\em recurrent} if $s\big|_u $  is infinite and {\em uniformly recurrent} if $s\big|_u$ is {\em syndetic}, i.e., there exists an integer $k$ such that in any factor of $s$ of length $k$ there is at least one occurrence of $u$. An infinite word $s$ is called {\it recurrent} (resp., {\it uniformly recurrent}) if each of its factors is recurrent (resp., uniformly recurrent). It is readily verified that  the word  $s$ is non-recurrent  (resp., non-uniformly recurrent) if and only if there exists a prefix of $s$ which is non-recurrent (resp., non-uniformly recurrent).

Let $s \in A^{\omega}$ and ${\cal S}$ denote the {\em shift operator}. The {\em shift orbit} of $s$ is the set $\orb(s)=\{{\cal S}^k(s) \mid k\geq 0\}$,  i.e., the set of all suffixes of $s$. The {\em shift orbit closure} of $s$ is the set $\{y \in A^{\omega} \mid \Ff y \subseteq \Ff s \}$. Two infinite words  $s$ and $t$ are in the same shift orbit if there exists an infinite word $z$ such that $s,t \in \orb(z)$, i.e., if and only if  either $s$ is a suffix of $t$ or $t$ is a suffix of $s$.

 An infinite word $s$ is called {\it periodic} if $s=u^\omega$ for some $u\in A^+,$ and is called {\it ultimately periodic} if $s=vu^\omega$ for some $v\in A^*,$ and $u\in A^+.$ As is well known,  an ultimately periodic word which is non-periodic is not recurrent (see, for instance, \cite[Lemma 1.4.4]{DV}). The word $s$ is called {\it aperiodic} if  $s$ is not ultimately periodic.

An infinite word $s$ is called {\it Lyndon} if there exists a linear order on $A$ with respect to which $s$ is lexicographically smaller than all its proper suffixes (see, for example,  \cite{SI}).  In particular, Lyndon words are not periodic, although they may be ultimately periodic, e.g., $ab^\omega.$ It is easy to see that each Lyndon infinite word begins with  an infinite number of distinct unbordered prefixes.

Given $k$ infinite words $x^{(1)},x^{(2)}, \ldots ,x^{(k)} \in A^\omega$ we consider the following {\em shuffle operator}~\cite{icalp}
$$ \Sh (x^{(1)},x^{(2)}, \ldots ,x^{(k)})\subseteq A^\omega$$ defined as  the collection of all infinite words $z$ for which there exists a factorization
$$z=\prod_{i=1}^\infty U_i^{(1)}U_i^{(2)}\cdots U_i^{(k)}$$
with each $U_i^{(j)} \in A^*$ and with
$x^{(j)}=\prod_{i=1}^\infty U_i^{(j)}$ for each $1\leq j\leq k.$

For all definitions and notation not explicitly given in the paper, the reader is referred to the books  \cite{AS,LO,LO2}.

\section{Prefixal factorizations}\label{bpcp}

We begin this section with an application of  the infinite Ramsey theorem \cite{GRS} which is relevant to the coloring problem:

\begin{proposition}[see \cite{GRS} and \cite{schutz}]\label{Ram} Let  $x\in A^{\omega}$ and $c: \Ff^+x\rightarrow C$ be any finite coloring  of the non-empty factors of $x.$ Then there exists a suffix $x'$ of $x$ which admits a monochromatic factorization relative to $c.$ 
\end{proposition}

\begin{proof} Let $\Sigma_2(\nats_+)$ denote the set of all two element subsets of the positive integers. Let $x= x_1x_2\cdots \in A^\omega$ and $c: \Ff^+x\rightarrow C$ be any finite coloring  of the non-empty factors of $x.$ Then $c$ induces a finite coloring $c': \Sigma_2(\nats_+)\rightarrow C$ by the following rule: Given $i<j,$ set $c'(\{i,j\})=c(x_i\cdots x_{j-1}).$
By  the infinite Ramsey theorem, there exists an infinite subset ${\cal N}=\{n_1<n_2<n_3<\cdots \}\subseteq \nats_+$ and $r\in C$ such that 
$c'(\{i,j\})=r$ for all $i,j\in {\cal N}$ with $i<j.$  Hence the suffix $x'=x_{n_1}x_{n_1+1}x_{n_1+2}\cdots$ of $x$ admits the following  monochromatic factorization relative to the coloring $c:$ 
\[x'=(x_{n_1}\cdots x_{n_2-1})(x_{n_2}\cdots x_{n_3-1})(x_{n_3}\cdots x_{n_4-1})\cdots. \qedhere \]
\end{proof}

The previous proposition suggests that given an infinite word $x\in A^\omega,$ a good coloring of its factors  relative to the coloring problem should depend in part on $x$ itself and not just on the set of its factors. With this in mind, it is natural to distinguish between prefixes and non prefixes of $x.$

\begin{definition} We say that an infinite word $s\in A^{\omega}$ admits  a {\em prefixal factorization}
if $s$ has a factorization
$$ s= U_1U_2\cdots $$
where each $U_i$, $i\geq 1$, is a non-empty prefix of $s$.
\end{definition}

For example, the {\em Fibonacci}  word $f= abaababaabaab\cdots$ fixed by the morphism $a\mapsto ab,$ $b\mapsto a$ can be written (uniquely) as a concatenation of $ab$ and $a$ and hence admits a prefixal factorization. Moreover, by applying the above morphism, one can produce an infinite number of distinct prefixal factorizations.

On the other hand there exist words which do not admit  prefixal factorizations, e.g., square-free words.
In fact, if a word $s\in A^\omega$ admits a prefixal factorization $s=U_1U_2\cdots,$
then $|U_i|\leq |U_{i+1}|$ for some $i\geq 1$ and hence $s$ must contain a square.

\noindent The following proposition indicates the relevance of prefixal factorizations to the coloring problem.

\begin{proposition}\label{prop:C} If $s\in A^\omega$ does not admit a prefixal  factorization, then $s\in {\cal P}_2$.
\end{proposition}

\begin{proof} It suffices to color each non-empty factor $u$ of $s$ by $c(u)=1$, if $u$ is a prefix of $s$ and $c(u)= 0$, otherwise.
\end{proof}

\noindent Thus ${\cal P}_2$ includes all square-free words.

  \begin{lemma}\label{lemma:B} Let $x\in A^{\omega}$ be an infinite word having a prefixal factorization. Then the first letter of  $x$ is uniformly  recurrent.
\end{lemma}
\begin{proof} Let $x=U_1U_2\cdots $ be a prefixal factorization and let $a$ denote the first letter of $x$. Then every factor $u$ of $x$ of length $> |U_1|$  contains an occurrence of $a$. In fact, the first occurrence of $u$ cannot be contained in any $U_i$, $i\geq 1$, and hence this first occurrence of $u$ must either contain some $U_i$ or must overlap two adjacent $U_j$'s. In either case, $a$ occurs in $u$. Thus  the first letter of $x$ is uniformly recurrent.
\end{proof}

A word  $x\in A^{\omega}$ is of  {\em full complexity} if  $A^+= \Ff^+ x$.  We observe that a significant subclass of full complexity words is constituted by the so-called {\em normal words} (see, for instance, \cite{NI}). A word $x\in A^{\omega}$ is normal  if all words of $A^+$ of equal length occur in $x$ with equal (asymptotic) frequency.

\begin{corollary} If $x\in A^{\omega}$ is of full complexity and  $ \card(\alf x)>1$, then $x$ is ${\cal P}_2$.
\end{corollary}
\begin{proof} Since $x$ is an infinite word of full complexity and $ \card(\alf x)>1$,  the first letter of $x$ cannot be uniformly recurrent; hence  by Lemma \ref{lemma:B} and Proposition \ref{prop:C}, one has  $x\in {\cal P}_2$.
\end{proof}

\begin{corollary} Almost all words $x\in A^{\omega}$, with $ \card(\alf x)>1$, belong to ${\cal P}_2$.
\end{corollary}
\begin{proof} It is well known (see, for instance, \cite[Theorem 10.1.6]{AS})  that relative to the standard Bernoulli measure on $A^{\omega}$, almost all $x\in A^{\omega}$ are of full complexity; hence  the result follows from the preceding corollary.
\end{proof}

\begin{lemma}\label{lemma:preffac} Let  $s\in A^{\omega}$. Then $s$  admits a prefixal factorization if and only if  $s$ begins with only finitely many unbordered prefixes.
\end{lemma}
\begin{proof} Suppose that $s$ admits a prefixal factorization $s= U_1U_2\cdots$. Any prefix $u$ of $s$ of length greater than $|U_1|$ is bordered. Indeed,  we can write $u = U_1\cdots U_ku'$ for a suitable $k\geq1$ and $u'$ a proper prefix of $U_{k+1}$ and then of $s$. If $u'\neq \varepsilon$, then $u'$, as $|u'|<|u|$,  is a proper prefix and suffix of $u$, i.e.,  a border of $u$. If $u'=\varepsilon$, then $U_k$ is a border of $u$. Hence, $u$ is bordered.

 Let us now prove the converse. Let $u$ be the longest unbordered prefix of $s$, and put $m=|u|$. Let $S$ be the set of all non-empty prefixes of $s$ of length at most $m$. Then  every prefix $v$ of $s$  can be written as a product  $v=v_1v_2\cdots v_k$ where each $v_i$, $1\leq i \leq k$, is in $S$.
  This is clear if  $|v|\leq m$. If $|v|>m$, then $v$  is bordered so we can write $v=v'v''$ where $v'$ and $v''$ are both non-empty prefixes of $s$. By  induction on the length of $v$, each of $ v'$ and $v''$ is a product of elements of $S$, whence $v$ is a product of elements of $S$.
By the pigeonhole principle, infinitely many prefixes of $s$ begin with the same $U_1$ in $S$. Of those infinitely many begin with the same $U_1U_2$ with $U_2$ in $S$. Of those infinitely many begin with the same $U_1U_2U_3$ with $U_3$ in $S$. Continuing we get $s=U_1U_2U_3 \cdots$ where each $U_i$, $i\geq 1$,
is in $S$ (note that this proof is just an application of the usual K\"onig infinity lemma).
\end{proof}

\noindent Thus we obtain:

\begin{corollary}If $x\in A^\omega$ begins with  an infinite number of unbordered prefixes, then $x\in {\cal P}_2.$
\end{corollary}

\noindent   Lyndon infinite words begin with infinitely many unbordered prefixes, so they belong to   ${\cal P}_2$.

\begin{remark} If $x\in A^{\omega}$ is an aperiodic uniformly recurrent word, then there exists an element $y$ in its shift orbit closure ${\cal S}_x$ which belongs to ${\cal P}_2$. In fact, as $x$ is uniformly recurrent,  one has ${\cal S}_x= \{z \in A^{\omega} \mid \Ff z = \Ff x\}$.
For any linear order on $A$,  let $y$ be the lexicographically smallest element of ${\cal S}_x$ relative to this order. For any $k>0$, ${\cal S}^k(y) \in {\cal S}_x$. Hence, $y$ is lexicographically smaller than all its suffixes, i.e., $y$ is a Lyndon word and hence is  in ${\cal P}_2$.
\end{remark}

\begin{example} The {\em regular paper-folding}  word (see, for instance, \cite{AS})
\[x= 00100110001101100010\cdots\]
is the limit, as $n\rightarrow \infty$,  of the sequence of words $(w_n)_{n>0}$ defined recursively as follows:
$w_1=0$ and  $w_{n+1}= w_n \  0 \  ({\overline w_n})^{\sim}$, where ($^{\sim} $) is the {\em reversal operator}
and $(^{-})$
 the automorphism of $\{0,1\}^*$ defined by ${\bar 0} =1$ and ${\bar 1}= 0$.
 One has  $|w_n|= 2^n-1$ and for any  $0<k < |w_n|$ the prefix $u$ and the suffix  $v$ of $w_n$ of length $k$ are such that $|u|_0>k/2$ and $|v|_0\leq k/2$. Then $x$ has infinitely many unbordered prefixes, so that  $x\in {\cal P}_2$.
\end{example}

\begin{proposition}\label{basic} Let  $s\in A^{\omega}$ be a non-periodic word satisfying either  of the following two conditions:
\begin{enumerate}
\item[(i)] There exists a non-negative integer $t$ such that for each prefixal  factorization $s= U_1U_2 \cdots$ we have  $\liminf_{n\rightarrow \infty}|U_n|\leq t.$
\item[(ii)] Given any prefixal  factorization $s= U_1U_2  \cdots$ there exist  $i,j\geq 0$ such that $U_i$ and $U_j$ terminate with different letters.
\end{enumerate}
Then  $s \in {\cal P}$.
\end{proposition}

\begin{proof}First suppose $s$ satisfies condition (i) above. Then we  finitely color each non-empty factor $u$ of $s$ as follows:
$$ c(u)= \begin{cases}
|u|, & \text{if $u$ is a prefix of $s$  and $|u|\leq t $}; \\
t+1, & \text{if $u$ is a prefix of $s$ and $|u|\geq t+1$}; \\
0, & \text{otherwise}.
\end{cases}$$
It follows that if $s= U_1U_2\cdots $ is a monochromatic factorization of $s,$ then
each $U_i$ is a prefix and hence $0\leq c(U_i)\leq t$ for each $i\geq 1.$ It follows that $s=U_1^\omega,$ a contradiction. Next suppose $s$ satisfies condition (ii) above and take $b\notin A.$  Then we finitely color each non-empty factor $u$ of $s$ as follows: $c(u)$ is the last symbol of $u$ if $u$ is a prefix of $s,$ and $c(u)=b$ otherwise. Then clearly $s$ does not admit a monochromatic factorization.
\end{proof}

\noindent  Recall that a word $u\in A^\omega$ is called {\it overlap-free} if $u$ does not contain a factor of the form $vva$ where $a$ denotes the initial letter of $v.$ As an immediate corollary we deduce:

\begin{corollary} If  $s\in A^\omega$ is overlap-free, then $s\in {\cal P}$.
\end{corollary}

\begin{proof} Let $s= U_1U_2  \cdots$ be a prefixal factorization. Pick $i\geq 2$ such that $|U_i|\leq |U_{i+1}|.$ Let $a$ and $b$ denote the last letters of $U_{i-1}$ and $U_i$, respectively. Put $U_i=V_ib$ with $V_i\in A^*.$ Then $aV_ibV_ib$ is a factor of $s$. Since $s$ is overlap-free,  it follows that $a\neq b.$ The result now follows from item (ii) of Proposition~\ref{basic}.
\end{proof}

By Proposition \ref{basic}, one derives that for an overlap-free word $s$, we have $s\in{\cal P}_k$ with  $k=$ $\card(\alf s)+1.$ As is well known \cite{LO}, over a binary alphabet $\{a, b\}$ one can construct an infinite overlap-free word by considering the fixed point beginning with $a$ of the morphism $\mu: \{a, b\}^* \rightarrow \{a, b\}^*$, called {\em Thue's morphism}, defined by
$\mu(a)= ab, \mu(b)= ba$. In this way one obtains the famous  {\em Thue-Morse} word:
$  abbabaabbaababba\cdots $. Thus the  Thue-Morse word  belongs to the class ${\cal P}_3$.

\begin{corollary} 
 If an infinite word $s$ has  only finitely many distinct squares, then  $s$  is in ${\cal P}$.
 \end{corollary}
 \begin{proof}Let $n=\max\{|v| \mid  v^2 \in \Ff^+  s \}$. We first observe that $s$ is non-periodic  since it has only finitely many distinct squares. For each prefixal factorization of $s= U_1U_2U_3\cdots$   we have $|U_i|\leq n$  for infinitely many $i$. Indeed, otherwise, there exists  $j$  such that $n< |U_j|\leq |U_{j+1}|$, so that    $U_j^2$ is a factor of $s$, a contradiction. Hence, the result follows by item (i) of Proposition~\ref{basic}.
 \end{proof}
 
\noindent  In a similar way one can prove that an infinite word having only finitely many overlaps is in ${\cal P}$.

\begin{corollary}\label{nonrec} If $s\in A^\omega$ is not recurrent, then $s\in {\cal P}_k$
with $k= |p_s| +2$, where $p_s$ is the shortest non-recurrent prefix of $s$.
\end{corollary}

\begin{proof}
Since $s$ is non-recurrent, there exists a shortest prefix $ p_s$ which is non-recurrent, i.e., occurring in $s$ only finitely many times. If  $s= U_1U_2 \cdots $  is a prefixal factorization, then for all but finitely many $i$
we have $|U_i| < |p_s|$. The result now follows from item (i) of Proposition~\ref{basic}.
\end{proof}

\begin{proposition}\label{orbit} Let $x\in A^\omega.$ Then  $x \in {\cal P}$ if and only if $ax \in {\cal P}$ for each $a\in A.$
\end{proposition}

\begin{proof}First suppose $x\in {\cal P}$ and let $a\in A.$ Let $c': \Ff ^+x\rightarrow C'$ be a finite coloring of $\Ff^+x$ relative to which $x$ does not admit a monochromatic factorization. Let $*$ and $\diamondsuit$  be symbols not in $C'$ and set $C=C'\cup \{*, \diamondsuit \}.$ We define a coloring $c: \Ff ^+ax\rightarrow C$ as follows: If $u$ is a prefix of $ax$ and $ua \in \Ff^+ax$, then set $c(u)=c'(a^{-1}ua).$ If $u$ is a prefix of $ax$ and $ua \notin \Ff^+ax$, then set $c(u)=*.$ If $u$ is not a prefix of $ax,$ set $c(u)=\diamondsuit.$ Suppose that a monochromatic factorization of $ax$ exists. Then it is given by $ax=U_1U_2U_3\cdots$, where for each $i\geq 1,$ we have that $U_i$ is a prefix of $ax$ and $U_ia\in \Ff^+ax.$ Thus $c'(a^{-1}U_ia)=c'(a^{-1}U_ja)$ for each $i,j\geq 1.$ But this gives rise to the monochromatic factorization $x=a^{-1}U_1a\cdot a^{-1}U_2a\cdot a^{-1}U_3a \cdots,$ a contradiction. Thus $ax \in {\cal P}.$

Next suppose that $ax \in {\cal P}$ for each $a\in A.$ Then for each $a\in A$ there exists a finite coloring
$c_a:\Ff^+ax \rightarrow C_a$ relative to which $ax$ does not admit a monochromatic factorization.
Set $C=A\times \bigcup _{a\in A }C_a \cup\{*\}$ with  $* \not\in \bigcup _{a\in A }C_a$.  Then we define a coloring $c: \Ff^+x\rightarrow C$ as follows: Let  $u$ be a non-empty factor of $x,$ and let $a\in A$ denote the last symbol of $u.$ If $u$ is a prefix of $x$  then set $c(u)=(a, c_a(aua^{-1})).$ In all other cases set $c(u)=*.$ Thus if $x=U_1U_2U_3\cdots$ is a  monochromatic factorization relative to $c,$ then each $U_i$ is a prefix of $x$ terminating with  the same letter $a\in A,$ and $c_a(aU_ia^{-1})=c_a(aU_ja^{-1})$ for each $i,j\geq 1.$ This implies that
$ax=aU_1a^{-1}\cdot aU_2a^{-1}\cdot aU_3a^{-1}\cdots$ is a monochromatic factorization of $ax,$ a contradiction. Thus $x$ does not admit a monochromatic factorization with respect to the coloring $c,$ and hence $x\in {\cal P}.$
\end{proof}

\noindent Combining Proposition~\ref{prop:C} and Proposition~\ref{orbit} we obtain:

\begin{corollary}\label{lemma:N1}Let $x\in A^\omega$. Suppose there exists a positive integer $n$ such that for each $u\in A^n$ the word $ux$ does not admit a prefixal factorization. Then  $x \in {\cal P}$.
\end{corollary}

\noindent As a special case of the above we have:

\begin{corollary} \label{cor:N1} Let $x\in A^\omega.$ Suppose there exists a subset $A'\subseteq A$ with $\card(A')\geq 2,$ such that for each $b\in A'$ the word $x$ begins with an infinite number of prefixes of the form $Ub$ with $U$ a palindrome. Then  for each $a\in A$ we have that $ax$ does not admit a prefixal factorization and hence $x\in {\cal P}$.
\end{corollary}

\begin{proof}We note that for each letter $a\in A,$ the word $ax$ does not admit a prefixal factorization. In fact, by hypothesis, for each $a\in A,$ there exists  $b\neq a$ such that $ax$ begins with  an infinite number of distinct prefixes of the form $aUb$ with $U$ a palindrome. It follows that $ax$ begins with arbitrarily long unbordered prefixes and hence, by Lemma~\ref{lemma:preffac}, does not admit a prefixal factorization. The result now follows from Corollary~\ref{lemma:N1} taking $n=1.$
\end{proof}

\begin{remark}\label{rem:eli}  An infinite word $s$ over the alphabet $A$ is called {\em standard episturmian} if it is closed under reversal and every left special factor of $s$ is a prefix of $s$. It is well known that every non-periodic standard episturmian word satisfies the hypothesis of Corollary~\ref{cor:N1} (see, for instance \cite{DJP,JP}). Hence every non-periodic standard episturmian word is in ${\cal P}$ (a different and longer proof of this latter result is in \cite{adleplz}). It can be shown that any suffix of a standard episturmian word is also in ${\cal P}$ (see Section \ref{Sec:sw}). Combined with Proposition~\ref{orbit} it follows that each word which is  in the same shift orbit of a standard episturmian word belongs to ${\cal P}.$
\end{remark}

\begin{remark}\label{rem:pdw} The {\em period-doubling} word
 $$x=0100010101000100\cdots$$
 is  defined as the fixed point $\tau^{\omega}(0)$ of the morphism $\tau$ given by $0\mapsto 01,$ $1\mapsto 00$ (see, for instance, \cite{AS}).
 The word $x$ begins with  an infinite number of distinct prefixes of the form $U0$ and $U1$ with $U$ a palindrome. In fact,
 letting $a\in \{0, 1\}$, we see that  if $Ua$ is a prefix of $x$ with $U$ a palindrome, then $\tau(U)0$ is a palindromic prefix of $x$ immediately followed by ${\bar a}$. Thus for $a\in \{0, 1\}$, $ax$ does not admit a prefixal factorization. By Corollary \ref{lemma:N1}, it follows
 that $x\in {\cal P}$.
\end{remark}

\section{Invariance properties of ${\cal P}$}\label{sec:ipp}

The following proposition describes a fundamental  invariance property of ${\cal P}$  having many applications to  the coloring problem.
\begin{proposition}\label{invariance} Let $s$ and $t$ be infinite words and  $h$ be a morphism such that
$t= h(s)$.  If $t \in {\cal P}$, then $s \in {\cal P}$.
\end{proposition}
\begin{proof}  By hypothesis, there exists a finite coloring $c_t: \Ff^+ t \rightarrow C$ such that  no factorization $ t = V_1V_2  \cdots$
is monochromatic.
Let $e$ be a symbol not in $C$.  We  consider  the coloring map $c_s: \Ff^+ s\rightarrow C\cup\{e\}$ defined as follows:  for any  $U\in \Ff^+ s$, we set $c_s(U)= c_t(h(U))$ if $h(U)\neq \varepsilon$ and $c_s(U)= e$ otherwise.  It follows that no factorization of $s$ is monochromatic. In fact,  if $s = U_1U_2 \cdots $ is a monochromatic factorization (with respect to $c_s$), since $t= h(s)= h(U_1)h(U_2)\cdots, $ one has $c_s(U_i)\neq e$ for each $i$ and hence $h(U_1)h(U_2)\cdots $ would be a monochromatic factorization of $t$ (with respect to $c_t$), a contradiction.\end{proof}

\begin{remark}\label{rem:lcolo}  Following the proof of the preceding proposition, if $t\in {\cal P}_k$, then $s\in {\cal P}_{k+1}$. If the morphism $h$ is non-erasing, then $s\in {\cal P}_k$.
\end{remark}

\begin{remark}\label{rem:bin}  If ${\cal P}$ contains all binary non-periodic words, then it contains all non-periodic words over any finite alphabet $A.$ Indeed, if $x\in A^\omega$ is non-periodic, then  there exists $a\in A$ such that $h_a(x)\in \{0,1\}^\omega$ is non-periodic, where $h_a:A^*\rightarrow \{0,1\}^*$ is the morphism  defined as follows:  for each $t\in A,$ we set  $h_a(t)=1$ if $t=a$ and $h_a(t)=0$ if $t\neq a.$
\end{remark}

\noindent  Some applications of Proposition \ref{invariance} are given by the two following propositions:

\begin{proposition} Let  $t \in A^\omega$. For each $a \in A$, let $B_a$ be a finite set such that $B_a$ and $B_b$ are disjoint whenever $a\neq b$. For each $a\in A$, let  $x(a)$ be any infinite word over the alphabet  $B_a$  and $s$ be the word obtained from $t $ by replacing for each $a\in A$, the subsequence of $a$'s  in $t $ by $x(a)$. If  $t \in {\cal P}$, then
 $s \in {\cal P}$.
\end{proposition}
\begin{proof}
In fact, let  $B= \bigcup_{a\in A}B_a$ and $h$ be the morphism of $B^*$ to $A^*$ defined as follows: for  $b\in B_a$, we set  $h(b)= a$
for each $a\in A$.  Then $h(s)=t$ and the result follows from Proposition~\ref{invariance}.
\end{proof}

\begin{example} Let  $s$  be  obtained from the Thue-Morse  word $abbabaabbaababba \cdots$  by replacing the subsequence of $a$'s  by the Fibonacci word  $01001010010\cdots$ over  the alphabet $\{0, 1\}$, and the subsequence of $b$'s by the period-doubling word $232223232322 \cdots$ over  the alphabet $\{2, 3\}$. Since the Thue-Morse word is in ${\cal P}$, by the preceding lemma, the word
$s= 0231200221031230\cdots$ is in ${\cal P}$.
\end{example}

\begin{proposition}\label{prop:prop20} Let $z\in \Sh (x^{(1)},x^{(2)}, \ldots ,x^{(k)})$ with $x^{(i)}\in A_i^{\omega}$, $1\leq i \leq k$ and $A_i\cap A_j=\emptyset$ for $i\neq j$. If there exists $i$, $1\leq i \leq k$, such that $x^{(i)}\in {\cal P}$, then $z\in {\cal P}$.
\end{proposition}
\begin{proof} Let us suppose that  $x^{(i)}\in {\cal P}$ for a suitable $i$.  The word $z$ is over  the alphabet $A= \bigcup_{i=1}^k A_i$.
We consider the morphism $h: A^*\rightarrow A_i^*$ defined by  $h(a)= a$ if $a\in A_i$ and $h(a)= \varepsilon$, otherwise.
From the definition of $z$ one has $h(z)= x^{(i)}$. Since $x^{(i)}\in {\cal P}$, by Proposition~ \ref{invariance}, it follows $z\in {\cal P}$.
\end{proof}

 Let $t\in A^\omega$ be an infinite word having a  recurrent prefix $u\in A^+$. Let $\mathcal{R}_u(t)$ denote the set of all {\em return words} to $u$ in $t,$ i.e., the set of all $v\in A^+$ such that  $vu$ is a factor of $t$ which begins and ends with $u$ and $|vu|_u=2$; $vu$ is called a {\em complete return} to $u$.
 The  prefix $u$ is uniformly recurrent if and only if  $\mathcal{R}_u(t)$ is finite. A recurrent  word $t$ is uniformly recurrent if and only if $\mathcal{R}_u(t)$ is finite for all non-empty prefixes $u$ of $t$.

 If $u\in A^+$ is a uniformly recurrent  prefix of $t$,  one can induce  a linear ordering on the set $\mathcal{R}_u(t)$  as follows: for $v, v' \in \mathcal{R}_u(t)$ we set
 $v<v'$ if
 the first occurrence of $v$ in $t$ precedes that of $v'.$
Let $A'=\{1,2,\ldots ,\card(\mathcal{R}_u(t))\},$
and let $\sigma :A'\rightarrow \mathcal{R}_u(t)$ be the unique order-preserving bijection which induces an isomorphism, still denoted by $\sigma$, of $A'^*$ and $(\mathcal{R}_u(t))^*$.
The sequence $t$ may be written uniquely as a concatenation
$t=v_1v_2v_3\cdots $ with $v_i\in \mathcal{R}_u(t).$

F.  Durand  \cite{Du} defines the {\it derived word} of
$t$ with respect to $u$ as
$$\mathcal D_u(t)=\sigma ^{-1}(v_1)\sigma ^{-1}(v_2)\sigma ^{-1}(v_3)\cdots $$
regarded as an infinite word over  the alphabet $A'.$

\begin{corollary}\label{cor: cprw}Let $u\in A^+$ be a uniformly recurrent prefix of a  word $t\in  A^\omega$.  Then   $t \in {\cal P}$ if and only if $\mathcal{D}_u(t) \in {\cal P}$.
\end{corollary}

\begin{proof} One direction is an immediate consequence of Proposition~\ref{invariance}. In fact, since  $t$ is the morphic image under $\sigma$ of  $\mathcal{D}_u(t)$, it follows that if  $t \in {\cal P}$, then $\mathcal{D}_u(t) \in {\cal P}$.
Conversely, suppose  $\mathcal{D}_u(t) \in {\cal P}$.
More precisely, there exists a finite coloring $c':  \Ff ^+\mathcal{D}_u(t) \rightarrow C'$ with respect to which no factorization of $\mathcal{D}_u(t) $ is monochromatic. Without loss of generality we can assume that
$C'\cap \{-1,0,1,2,\ldots, |u|-1\}=\emptyset.$ We now define a finite coloring $c: \Ff ^+t \rightarrow C' \cup \{-1,0,1,2,\ldots, |u|\}$ as follows: for $z\in \Ff ^+t,$  set
$$ c(z)= \begin{cases}
|z|, & \text{if  $z$ is a prefix of $t$ with $|z|< |u|$};\\
c'(\sigma^{-1}(z)), & \text{if $z$ is a prefix of $t$ with $|z|\geq |u|$ and $zu\in  \Ff ^+t $};\\
0, & \text{if $z$ is a prefix of $t$ with $|z|\geq |u|$ and $zu\notin  \Ff ^+t $};\\
-1, & \text{if $z$ is not a prefix of $t$}.
\end{cases}$$
 We note that if $z\in \Ff ^+t$ begins with  $u$ and $zu \in \Ff ^+t$, then $z$ is a concatenation of  return words  to $u$ and hence is in the image of $\sigma.$
Now, since $t$ is not periodic, a monochromatic factorization of $t$ corresponds to a factorization of $t=U_1U_2U_3\cdots$ in which each $U_i$ is a prefix of $t$ beginning with $u.$ Hence for each $i\geq 1,$ we have  $U_i=\sigma (V_i)$ for some factor $V_i$ of   $\mathcal{D}_u(t) $ and $c(U_i)=c'(V_i).$ This gives the monochromatic factorization $\mathcal{D}_u(t) =V_1V_2V_3\cdots,$ a contradiction.
\end{proof}

Since a word $t$ is uniformly recurrent  if and only if  the set $\mathcal{R}_u(t)$ is finite for any non-empty prefix $u$ of $t$,
from the preceding result one obtains:

\begin{corollary}\label{derived}Let $u\in A^+$ be any prefix of a uniformly recurrent word $t\in  A^\omega.$   Then  $t \in {\cal P}$ if and only if $\mathcal{D}_u(t) \in {\cal P}$. \end{corollary}

\section{Non-uniformly recurrent words}\label{sec:nur}

 We shall now prove that the coloring problem has a positive answer for any infinite word which is not uniformly recurrent (cf. Theorem~\ref{thm:notunire}).
We need some preparatory lemmas.

Let $Z$  be any symbol, and let $W$ consist of $Z$  together with all infinite words which take on finitely many values. Define a map $$S: W\rightarrow W$$  as follows:
If $x$ is an infinite word  and  $a$ the first symbol of $x$, then $S(x)$  is equal to ${\cal D}_a(x)$  if $a$  is uniformly recurrent in $x$, and equal to $Z$ otherwise. Moreover, we set  $S(Z)=Z$. From Lemma \ref{lemma:B} one has that if $x$ has a prefixal factorization, then $S(x) \neq Z$.

\begin{lemma}\label{lemma D} Suppose $x$ and $S(x)$ are infinite words. If $S(x)$ is in ${\cal P}$, then $x$ is in ${\cal P}$.
\end{lemma}
\begin{proof} Since $S(x)$ is an infinite word, $S(x)= {\cal D}_a(x)$ where $a$, the first letter of $x$, is uniformly recurrent.  As ${\cal R}_a(x)$ is finite and $S(x) \in {\cal P}$,  by Corollary~\ref{cor: cprw}, it follows that $x\in {\cal P}$.
\end{proof}

\begin{lemma}\label{lemma:A} If $x$ is not uniformly recurrent, then there exists $n\geq  1$  such that
$S^n(x)=Z$.
\end{lemma}
\begin{proof}
Let us  first prove that for any $n\geq 1$ if $S^n(x)\neq Z$, then $S^n(x)= {\cal D}_u(x)$ for a suitable prefix of $x$ of length $|u|\geq n$.
The proof is by induction on the integer $n$.  For the base of the induction, we observe that  if $S(x)\neq Z$, then the first letter $a$ of $x$ is uniformly recurrent and $S(x)= {\cal D}_a(x)$. Let us suppose $n>1$ and $S^n(x)\neq Z$. One has
$$  S^n(x) =  {\cal D}_{1}(S^{n-1}(x)),$$
where   the first letter $1$ of  $S^{n-1}(x)$ is  uniformly recurrent in $S^{n-1}(x)$.
By induction, we have $S^{n-1}(x)=  {\cal D}_u(x)$ with $u$ the prefix of $x$ of length $|u|\geq n-1$. Hence, one has  $S^n(x) =  {\cal D}_{1}( {\cal D}_u(x))$.  From a general result on derived sequences  (see \cite[Proposition 2.6]{Du}) one has that
${\cal D}_{1}( {\cal D}_u(x))=  {\cal D}_w(x)$ where $w$ is a prefix of $x$ of length $|w|>|u|$.

Let $p$ be the shortest prefix of $x$ which is not uniformly recurrent. Since ${\cal D}_w(x)$ is not defined for  $|w|\geq |p|$, it follows that there must exist an integer $n\geq 1$ for which $S^n(x)= Z$.
\end{proof}

\begin{theorem}\label{thm:notunire} Let $x$ be a non-uniformly recurrent infinite word. Then $x$ is in ${\cal P}$.
\end{theorem}
\begin{proof} By Lemma \ref{lemma:A},  there exists a least positive integer $n \geq 1$ such that $S^n(x)=Z$. Then by Lemma~\ref{lemma:B}, $S^{n-1}(x)$ is an infinite word not admitting a prefixal factorization. By Proposition~\ref{prop:C}, $S^{n-1}(x)$ is in ${\cal P}$. Finally, by iteration of Lemma~\ref{lemma D}, we deduce that $x$ is in ${\cal P}$.
\end{proof}

\begin{proposition}\label{prop:numcol} Let   $x$ be a non-uniformly recurrent word and $p$ be the shortest prefix which is not uniformly recurrent in $x$. Then $x \in {\cal P}_k$ with
$k \leq  |p|+2$.
\end{proposition}
\begin{proof}
Since $x$ is not uniformly recurrent, by Lemma \ref{lemma:A},  there exists an integer $n\geq 1$, such that $S^n(x)=Z$. If $n=1$, then $|p|=1$ and by Lemma \ref{lemma:B},   $x$ does not have  a prefixal factorization, so that $x\in {\cal P}_2$. Let us suppose $n>1$. One has  $S^{n-1}(x)= D_u(x)$ for a suitable prefix $u$ of $x$ such that $|u| <|p|$. As  $S^n(x)=Z$,  the first letter of $S^{n-1}(x)$ is not uniformly recurrent
in $S^{n-1}(x)$. Thus $D_u(x)\in {\cal P}$ with a set $C'$ of colors having only two elements.
 From the proof of Corollary~\ref{cor: cprw},   $x$ is in ${\cal P}$ with a set  of colors
 $C= C' \cup \{-1,0, 1, 2, \ldots, |u|-1\}$. Hence, $\card(C) = \card(C') + |u|+1\leq  |p| +2$.
\end{proof}

\begin{example}\label{ex:luca}  Let $W_i$, $i\geq 1$, be the sequence of finite words recursively defined as follows: $W_1= ab$, $W_k = W_1\cdots W_{k-1}a$ for all $k>1$. Thus $W_2= aba$, $W_3= ababaa$, $W_4= ababaababaaa$,  etc.
Let $L\in \{a, b\}^{\omega}$ be the infinite word
 $$ L= W_1W_2 \cdots W_n \cdots.$$
 By construction, the word $L$ is recurrent. Moreover, any $W_k$, $k\geq 1$,  begins with  $a$, so that $W_k$ is a prefix of $L$. Since $W_k$ terminates with $ba^{k-1}$, one has that neither  $b$ nor  $ab$ occurs in $L$ with bounded gaps. Hence the prefix $p=ab$ is not uniformly recurrent.

 The  return words (see the discussion following Proposition \ref{prop:prop20}) of $a$ in $L$ are $ab$ and $a$, i.e., ${\cal R}_a = \{ab, a \}$. If $\sigma: \{1,2\}\rightarrow \{ab, a \}$, the derived
 sequence ${\cal D}_a(L)$ is given by
 $${\cal D}_a(L) = 1121122112211222\cdots$$
 The first letter $1$ of ${\cal D}_a(L)$ is not uniformly recurrent in ${\cal D}_a(L)$, so that $S^2(L)= Z$ and   ${\cal D}_a(L)$ does not admit a prefixal factorization. Thus the coloring problem has a positive solution for ${\cal D}_a(L)$  using only two colors.
 Letting  $c'$ denote this coloring, one has for any factor $z$ of ${\cal D}_a(L)$,  $c'(z)= W$  if $z$ is a prefix of ${\cal D}_a(L)$ and  $c'(z)=B$, otherwise.

 By using the construction given by Corollary~\ref{derived}, we deduce that  $L \in {\cal P}_4$.
 Indeed, the coloring map $c$ is defined as follows: for any $z\in L$, $c(z)= -1$ if $z$ is not a prefix of $L$,
 $c(z)= 0$ if $z$ is a prefix of $L$ and $za \not\in \Ff L$, $c(z)= c'(\sigma^{-1}(z))$ if  $z$ is a prefix of $L$ and $za\in \Ff L$. The set of colors is then $\{-1, 0, W, B\}$.
 \end{example}

 Since a non-recurrent word is non-uniformly recurrent, one has by Proposition~\ref{prop:numcol}, that a non-recurrent word $x$ is in ${\cal P}_k$  with $k \leq |p|+2$, where $p$ is the shortest non-uniformly recurrent prefix of $x$. This number is less than or equal to that given by Corollary~\ref{nonrec}, since $p$ is shorter than or equal to the shortest non-recurrent prefix.  In any case this bound can be very large; however, as we shall see shortly,  in the case of non-recurrent words   under suitable hypotheses
 the number of colors can be reduced to $4$ (cf. Proposition~\ref{prop:four}).

Let $s$ be an infinite word over the alphabet $A$. For any integer $n\geq 0$ we let
$s_{[n]}$ denote  the prefix of $s$ of length $n$. Moreover, for any letter $a\in A$ we set
$$ f'_a(s) = \liminf_{n\rightarrow \infty} \frac{|s_{[n]}|_a}{n}, \ \  f''_a(s)=  \limsup_{n\rightarrow \infty}\frac{{|s_{[n]}}|_a}{n}.$$
 If $f'_a(s)=f''_a(s)$,  then the limit  $  \lim_{n\rightarrow \infty}\frac{{|s_{[n]}}|_a}{n}$ exists, is denoted
 $f_a(s)$, and is called the {\em frequency of the letter} $a$ in $s$.
The following lemma will be useful in the following.

\begin{lemma}\label{lemma:alpha} Let $s\in A^{\omega}$ be an infinite word having a factorization
$$s = V_1V_2\cdots V_n \cdots, $$
where each $V_i$, $i\geq 1$, is a non-empty factor of $s$ of length $|V_i|\leq M$, with $M$ any fixed
positive integer. Let us define for all $a\in A$ and $i\geq 1$
$$f_a(V_i) = \frac{|V_i|_a}{|V_i|},$$
and
$m_a= \min\{f_a(V_i) \mid i\geq 1\}, \  M_a =\max\{f_a(V_i) \mid i\geq 1\}.$
For any $a\in A$
one has
$$m_a\leq  f'_a(s)\leq f''_a(s)\leq M_a.$$
\end{lemma}
\begin{proof} For a sufficiently large $n$ we can write for a suitable $k$
$$s_{[n]} = V_1\cdots V_kV'_{k+1},$$
with $V'_{k+1}$ a prefix of $V_{k+1}$ and $n= \sum_{i=1}^k|V_i| + |V'_{k+1}|$.
Thus as $|V'_{k+1}|\leq |V_{k+1}|\leq M$, one has
$$|s_{[n]}|_a= \sum_{i=1}^k|V_i|_a+|V'_{k+1}|_a\geq m_a \sum_{i=1}^k|V_i|\geq m_a(n-M).$$
From this it trivially follows that $f'_a(s)\geq m_a$. In a similar way one has
$$|s_{[n]}|_a\leq M_a \sum_{i=1}^k|V_i|+ |V'_{k+1}|_a \leq M_a(n-|V'_{k+1}|)+ |V'_{k+1}|.$$
Since $|V'_{k+1}|\leq M$, one derives
$$|s_{[n]}|_a \leq M_an +(1-M_a)M,$$
so that $f''_a(s) \leq M_a$. From this the result follows.
\end{proof}

 \begin{proposition}\label{prop:four} Let $s\in A^{\omega}$ be an infinite word such that there exists a letter $a\in A$ for which either $f'_a(s) < f''_a(s)$ or  $f'_a(s) = f''_a(s) = f_a(s)$ is an irrational number.  If there exists a positive integer $M$ such that for any prefixal factorization $s= U_1U_2\cdots U_n\cdots $ one has $|U_i|<M$, $i\geq 1$, then $s \in {\cal P}_4$.
\end{proposition}
\begin{proof} For any non-empty factor $V$ of $s$ we set $f_a(V)= |V|_a/|V|$. We consider the
coloring map $c: \Ff^+(s)\rightarrow \{0, 1, 2, 3\}$ defined as follows. For any $V\in \Ff^+(s)$: $$ c(V)= \begin{cases}
                 0,   & \mbox{if  $V$ is not a prefix of $s$};\\
                 1,   & \mbox{if  $V$ is a prefix of $s$ such that $|V|<M$ and $f_a(V) > f'_a(s)$};\\
                 2,   & \mbox{if $V$ is  a prefix of $s$ such that $|V|<M$ and $f_a(V) \leq  f'_a(s)$}; \\
                 3, & \mbox{if $V$ is  a prefix of $s$ and $|V|\geq M$.}
                 \end{cases}$$
 Let $s= V_1V_2\cdots V_n \cdots$ be any factorization of $s$ in non-empty factors and suppose that the factorization is monochromatic, i.e., for all $i\geq 1$, $c(V_i) = x\in \{0, 1, 2, 3\}$.
 One has that   $x \neq 0$ since $V_1$ is a prefix of $s$. If $x=3$, then all $V_i$, $i\geq 1$, are prefixes of $s$ of length greater than or equal to $M$, which contradicts the assumption made on $s$.  If $x=1$, then one has that $m_a > f'_a(s)$. By Lemma~\ref{lemma:alpha}, $f'_a(s)\geq m_a$, so we reach a contradiction. Let $x=2$. If $f'_a(s) < f''_a(s)$, then one has $M_a\leq f'_a(s) < f''_a(s)\leq M_a$ a contradiction.

 Let us suppose  $f'_a(s) = f''_a(s)= f_a(s)$ is an irrational number.  If  $f_a(V) \leq  f'_a(s)$, then, as  $f_a(V)$ is a rational number, one must have  $f_a(V) <  f'_a(s)$. Hence, it follows $M_a < f_a(s) \leq M_a$ a contradiction.
\end{proof}

\begin{corollary}Let $s\in A^{\omega}$ be a non-recurrent infinite word such that there exists a letter $a\in A$ for which either $f'_a(s) < f''_a(s)$ or  $f'_a(s) = f''_a(s) = f_a(s)$ is an irrational number. Then $s \in {\cal P}_4$.
\end{corollary}
\begin{proof} Let $p_s$ be the shortest non-recurrent prefix  and set $M= |p_s|$. For any prefixal factorization $s= U_1U_2\cdots U_n\cdots $ one must have  $|U_i|<M$, $i\geq 1$. Indeed, otherwise any $U_i$ has  $p_s$ as a prefix and $p_s$ is recurrent, a contradiction. From the preceding proposition the result follows.
\end{proof}

 \begin{example}  Let $s$ be the infinite word $s= baabaaf$, where $f$ is the Fibonacci word. The word $s$  is not recurrent and the shortest prefix which is not recurrent is $p_s= baabaaa$. One has that
 $\lim_{n\rightarrow \infty} |s_{[n]}|_a/n  = g-1$, where $g= \frac{1+ \sqrt 5  }{2}$ is the {\em golden ratio}.
 \end{example}

\section{Sturmian words}\label{Sec:sw}

A word $s\in \{a,b\}^{\omega}$ is called {\it Sturmian} if  it is aperiodic and {\em balanced}, i.e., for all factors $u$ and $v$ of $s$ such that $|u|=|v|$ one has
$$ | |u|_x-|v|_x| \leq 1, \ x\in \{a,b\}.$$

\begin{definition} We say that a Sturmian word is of type $a$ (resp.,  $b$) if it contains the factor $aa$  (resp., $bb$).
\end{definition}
\noindent Clearly a Sturmian word is either of type $a$ or of type $b$, but not both.

 Alternatively, a binary infinite word $s$ is Sturmian if  $s$ has a unique left (or equivalently right) special factor  of length $n$  for each  integer $n\geq 0$. In terms of factor complexity, this is equivalent to saying  that  $\lambda_s(n)= n+1$ for $n\geq 0$. As a consequence one derives that  a Sturmian word $s$ is {\em closed under reversal}, i.e., if $u$ is a factor of $s$, then so is its reversal $u^{\sim}$ (see, for instance, \cite[Chap. 2]{LO2}).

 A Sturmian word $s$ is called {\em standard} (or {\em characteristic})  if  all its prefixes are left special factors of $s$.
  As is well known (see, for instance, \cite[Chap. 2]{LO2}), for any Sturmian word $s$ there exists a standard Sturmian word $t$ such that
 $\Ff s = \Ff t$.

\begin{definition} Let $s\in \{a,b\}^{\omega}$ be a Sturmian word, $w$ a non-empty factor  of $s$, and  $z\in \{a, b\}$. We say $w$ is {\em rich} in the letter $z$ if there exists a factor $v$ of $s$ such that $|v|=|w|$ and $|w|_z>|v|_z$.
\end{definition}
From the aperiodicity and the balance property of a Sturmian word,  one easily derives that any non-empty factor $w$ of a Sturmian word $s$ is  rich either in the letter $a$ or in the letter $b$,  but not in both  letters.
Thus one can introduce, for any given Sturmian word $s$,
a map $$r_s :  \Ff ^+ s  \rightarrow \{a, b\}$$ defined as follows: for any non-empty factor $w$ of $s$, we set
$r_s(w)= z\in \{a, b\}$ if  $w$ is rich in the letter $z$. Clearly, $r_s(w)=r_s(w^{\sim})$ for any $w\in \Ff ^+ s$.

For any letter $z\in \{a, b\}$ let
${\bar z}$ be  the complementary letter of $z$, i.e., ${\bar a}= b$ and ${\bar b}= a$.
\begin{lemma}\label{lemma:lastletter} Let $w$ be a non-empty right special (resp., left special) factor
of a Sturmian word $s$. Then $r_s(w)$ is equal to the first letter of $w$ (resp.,  $r_s(w)$ is equal to the last letter of $w$).
\end{lemma}
\begin{proof} Write $w=zw'$ with $z\in \{a, b\}$ and $w'\in \{a, b\}^*$. Since $w$ is a right special factor of $s$,  one has that $v=w'{\bar z}$ is a factor of $s$. Thus  $|w|= |v|$ and $|w|_z >|v|_z$, whence
$r_s(w)= z $. Similarly, if $w$ is left special,  one deduces that $r_s(w)$ is equal to the last letter of $w$.
\end{proof}

\subsection{Preliminary Lemmas}

In this section we develop some new combinatorial properties of Sturmian words which will be used to show that the coloring problem has a positive answer for Sturmian words.

\begin{lemma}\label{lemma:cp} Let $s$ be a non-periodic infinite  word  having a prefixal factorization
$$ s = U_1\cdots U_n \cdots.$$
If  $c$ is the first letter of $s$, then $U_1\not\in c^+$.
\end{lemma}
\begin{proof} Suppose that $U_1=c^p$, with $p\geq 1$.  Since $s$ is non-periodic and  $U_i$ is a prefix of $s$ for  $i\geq1$, there exists a minimal integer $j$
such that  $\card(\alf U_j) >1$.  Since $U_j$ is a prefix of $s$, we have   $U_1\cdots U_{j-1}U_j = U_j \xi$, with $\xi \in A^*$.  As $U_1\cdots U_{j-1} = c^q$ for a suitable $q\geq p$, it follows that
$\xi= c^q$ and $U_j\in c^+$, a contradiction.
\end{proof}

\begin{lemma}\label{theorem:primo} Let $s$ be a Sturmian word such that
$$ s= U_1U_2 \cdots U_n \cdots,$$
where  the $U_i$'s are non-empty factors of $s$. If $r_s(U_i)= r_s(U_j)$ for all  $i, j \geq 1$, then
the sequence $|U_i|$, $i\geq 1$,  is unbounded.
\end{lemma}
\begin{proof} Suppose to the contrary that  for some positive integer $M$ we have that  $|U_i|\leq M$ for each $i\geq 1$. This
implies that the number of distinct $U_i$'s in the sequence $(U_i)_{i\geq 1}$ is finite, say $t$.
Let $ r_s(U_i)=x\in \{a, b\}$ for all $i\geq1$ and set for every  $i\geq 1$:
$$ f_x(U_i) = \frac{|U_i|_x}{|U_i|}.$$
Thus $\{f_x(U_i) \mid i\geq 1\}$ is a finite set of at most $t$ rational numbers. We set $m_x= \min\{f_x(U_i) \mid i\geq 1\}$.

 As is well known \cite [Proposition  2.1.11] {LO2},
 the frequency $f_x(s)$ of the letter $x$ in $s$ exists and is an irrational number:
 $$f'_x(s)= f''_x(s)=  f_x(s)=\lim_{n\rightarrow\infty} \frac{|s_{[n]}|_x}{n}.$$

Let us now prove that  $m_x >f_x(s)$. From \cite[Proposition 2.1.10] {LO2} one derives that for
all $V\in \Ff s$
$$  |V|f_x(s) -1 < |V|_x <   |V|f_x(s)+1.$$
Now $U_i$ is rich in the letter $x$  for all $i\geq 1$,  so that there exists $V_i \in \Ff s$ such that
$|U_i|= |V_i|$ and $|U_i|_x>|V_i|_x$.
From the preceding inequality one has
$$|U_i|_x = |V_i|_x+1>  |V_i|f_x(s)=  |U_i|f_x(s),$$
so that for all $i\geq 1$, $f_x(U_i)> f_x(s)$, hence $m_x >f_x(s)$.
By Lemma~\ref{lemma:alpha},   one would have $f_x(s)\geq m_x$
a contradiction.
\end{proof}

In the following we shall consider the Sturmian morphism $R_a$, that we simply denote $R$,
defined as follows:
\begin{equation}\label{morf:erre}
R(a)= a \  \mbox{and} \  R(b)=ba.
\end{equation}

For any finite or infinite word $w$, let $\Pre w$  denote  the set of all  prefixes of $w$.

\begin{lemma}\label{lemma:onetwo}
 Let $s$ be a Sturmian word and  $t\in \{a,b\}^{\omega}$ such that  $R(t)=s$. If either
 \begin{itemize}

\item [1)] the first letter of $t$ (or, equivalently, of $s$) is $b$

\vspace{2mm}

  or

  \vspace{2 mm}

\item[ 2)] the Sturmian word $s$ admits a prefixal factorization:
$$s = U_1\cdots U_n\cdots,$$
where each $U_i$, $i\geq 1$, terminates with the letter $a$ and  $r_s(U_i)=r_s(U_j)$ for all $i,j \geq 1$,
\end{itemize}
then $t$ is also Sturmian.
\end{lemma}

\begin{proof} Let us prove that in both  cases $t$ is balanced. Suppose to the contrary that $t$ is unbalanced.
Then (see \cite[Proposition 2.1.3]{LO2}) there would exist $v$ such that
$$ava, bvb \in \Ff t.$$
Thus
$$aR(v)a, \ baR(v)ba \in \Ff s.$$
If $ava \not\in \Pre t$, then  $t= \lambda ava \mu$, with $\lambda\in \{a, b\}^+$ and $\mu\in  \{a,b\}^{\omega}$. Therefore $R(t)= R(\lambda)R(ava)R(\mu)$. Since the last letter of $R(\lambda)$ is $a$, it follows that  $aaR(v)a \in \Ff s$. As $baR(v)b\in \Ff s$, we reach a contradiction with the balance
property of $s$. In  case 1), $t$ begins with the letter $b$, so that $ava \not\in \Pre t$ and then $t$ is balanced.
In  case 2) suppose that $ava\in \Pre t$. This implies that $aR(v)a \in \Pre s$. From Lemma \ref{theorem:primo}, in the prefixal factorization of  $s$ there exists an integer $i>1$ such that $|U_i|> |aR(v)a|$.
Since $U_{i-1}$ terminates with $a$ and $U_{i-1}U_i \in\Ff s $, it follows that $aaR(v)a \in \Ff s$ and this
again contradicts  the balance property of $s$. Hence, $t$ is balanced.

Trivially, in both cases $t$ is aperiodic, so that $t$ is Sturmian.
\end{proof}
Let us remark that, in general, without any additional  hypothesis, if $s=R(t)$, then $t$ need not  be Sturmian. For instance, if $f$ is the Fibonacci word $f= abaababaaba\cdots$, then $af$ is also a Sturmian word. However, it is readily verified that in this case
the word $t$ such that $R(t)=s$ is not balanced, so that $t$ is not Sturmian.

\subsection{Main results}

\begin{theorem}\label{theorem:basic1} Let $s$ be a Sturmian word  having a prefixal factorization
$$ s = U_1U_2 \cdots.$$
Then there exist integers $i,j \geq 1$ such that
$r_s(U_i)\neq r_s(U_j)$.
\end{theorem}

The key idea of the proof consists of  showing that if one supposes to the contrary  that  a Sturmian word $s$ has a factorization $s= U_1U_2\cdots$ in equally rich prefixes, then
there would exist another Sturmian word $t$ which can be factorized in equally rich prefixes $t= V_1V_2\cdots $ but with $|V_1|<|U_1|$. In this way one reaches a contradiction. We need the following two lemmas.

\begin{lemma}\label{prop:twocases} Let $s$ be a Sturmian word of type $a$ having a prefixal factorization
$$ s = U_1\cdots U_n \cdots,$$
 such that $r_s(U_i)=r_s(U_j)$ for all $i,j\geq 1$.
Then one of the  following two properties holds:
\begin{itemize}
\item[i)] All $U_i$, $i \geq 1$, terminate with the letter $a$.
\item[ii)] All  $U_ia$, $i\geq 1$, are prefixes of $s$.
\end{itemize}
\end{lemma}
\begin{proof} The result is clear if $s$ begins with $b$ as in this case each $U_i$ must terminate with $a.$ So suppose $s$ begins with  $a$ and some $U_i$ terminates with  $b.$ Let ${\cal I}$ denote the set of $j\geq 1$ such that $U_j$ terminates with $b,$ and let $i$ be the least element of ${\cal I}.$ In case $i\geq 2$ we have $aU_i \in \Ff^+s$, and hence $r_s(U_i)=b.$ Thus for all $j\geq 1,$ we have $U_jb \notin \Ff^+s$ and hence for all $j\geq 1$, $U_ja \in \Pre s.$  So suppose $i=1.$ If there exists $j\geq 2$ with $j\notin {\cal I}$ and $j-1\in {\cal I},$ then $r_s(U_j)=a.$ This implies that $aU_1 \notin \Ff^+s$ and that $U_k$ terminates with $a$ whenever $k\geq j.$  But by Lemma  \ref{theorem:primo}, $U_1$ is a prefix of some $U_k$ with $k>j,$ and hence $aU_1 \in \Ff^+s,$ a contradiction.
Thus in case $i=1,$ every $j\geq 1,$ belongs to ${\cal I},$ from which it follows that for every $j\geq 1,$  $U_jb \notin \Ff^+s$ and hence for all $j\geq 1,$ $U_ja \in \Pre s.$ \end{proof}

\begin{lemma}\label{prop:basic} Let $s$ be a Sturmian word  having a prefixal factorization
$$ s = U_1\cdots U_n \cdots,$$
 such that $r_s(U_i)=r_s(U_j)$ for all $i,j\geq 1$. Then there exists a Sturmian word $t$ having a prefixal factorization
$$t = V_1\cdots V_n \cdots,$$
such that  $r_t(V_i)=r_t(V_j)$ for all $i,j\geq 1$, and
$|V_1|<|U_1|$.
\end{lemma}
\begin{proof} We can suppose without loss of generality that $s$ is a Sturmian word of type $a$.
From Lemma~\ref{prop:twocases} either all  $U_i$, $i\geq 1$, terminate with  the letter $a$
or for all $i\geq 1$, $U_ia \in \Pre s$. We  consider two cases:

\vspace{2 mm}

\noindent
Case 1.  For all $i\geq 1$, $U_ia \in \Pre s$.

We can suppose that $s$ begins with the letter $a$. Indeed, otherwise, if the first letter of $s$ is $b$, then all $U_i$, $i\geq 1$, begin with the letter $b$ and, as $s$ is of type $a$, they must  terminate with the letter $a$. Thus the case that the first letter of $s$ is $b$ will be considered when we will analyze  Case 2.

We consider the injective endomorphism $L_a$ of $\{a, b\}^*$,  simply denoted $L$,  defined by
$$L(a)=a \ \mbox{and} \ \  L(b)=ab. $$ Since $s$ is of type $a$, the first letter of $s$ is $a$,  and $X=\{a, ab\}$ is a code having a finite deciphering delay \cite{BP}, the word $s$ can be uniquely factored as a product of  elements of $X$. Thus there exists a unique word $t \in \{a,b\}^{\omega}$ such that  $s = L(t)$. The following holds:
\begin{itemize}
\item[1.] The word $t$ is a Sturmian word.
\item[2.] For any $i\geq 1$ there exists a non-empty prefix $V_i$ of $t$ such that $L(V_i)= U_i$.
\item[3.] The word $t$ can be factorized as $ t= V_1\cdots V_n \cdots.$
\item[4.] $|V_1|<|U_1|$.
\item[5.] For all $i,j \geq 1$ one has  $r_t(V_i)= r_t(V_j)$.
\end{itemize}
Point 1. This is a consequence of the fact that $L$ is a standard Sturmian morphism \cite[Corollary 2.3.3]{LO2}.

\vspace{2 mm}

\noindent
Point 2.  For any $i\geq 1$, since $U_ia\in \Pre s$  and any pair $(c, a)$ with $c\in \{a, b\}$ is synchronizing for $X^{\infty}=X^*\cup X^{\omega}$  \cite{BP}, one has   $U_i\in X^*$, so that there exists $V_i\in \Pre t$ such that
$L(V_i)= U_i$.

\vspace{2 mm}

\noindent
Point 3. One has $L(V_1\cdots V_n\cdots) = U_1\cdots U_n \cdots= s = L(t)$. Thus
$t= V_1\cdots V_n \cdots$.

\vspace{2 mm}

\noindent
Point 4. By Lemma \ref{lemma:cp}, $U_1$ is not a power of $a$,  so that in $U_1$ there must be at least one occurrence of the letter $b$. This implies that $|V_1|<|U_1|$.

\vspace{2 mm}

\noindent
Point 5. We shall prove that  $r_t(V_i)= r_s(U_i)$ for all $i\geq 1$. From this  equality, one has that  $r_t(V_i)= r_t(V_j)$ for all $i,j \geq 1$.

Since $t$ is a Sturmian word, there exists $V'_i\in \Ff t$ such that
$$ |V_i|= |V'_i| \ \mbox{and} \ \mbox{either} \ |V_i|_a> |V'_i|_a \ \mbox{or} \  |V_i|_a< |V'_i|_a .$$
In the first case $r_t(V_i)= a$ and in the second case $r_t(V_i)= b$. Let us set
$$ F_i = L(V'_i).$$
Since $U_i = L(V_i)$, from the definition of the morphism $L$ one has
\begin{equation}\label{eq:uno}
|F_i|_a = |V'_i|_a+|V'_i|_b = |V'_i|, \ |F_i|_b= |V'_i|_b.
\end{equation}
\begin{equation}\label{eq:due}
|U_i|_a = |V_i|_a+|V_i|_b = |V_i|, \ |U_i|_b= |V_i|_b.
\end{equation}

Let us first consider the case $r_t(V_i)= a$, i.e., $|V_i|_a= |V'_i|_a+1$ and $|V_i|_b= |V'_i|_b-1$.
From the preceding equations one has
$$ |F_i|= |U_i|+1.$$
Moreover, from the definition of $L$ one has that  $F_i$ begins with the letter $a$. Hence,
$|a^{-1}F_i|= |U_i|$ and $|a^{-1}F_i|_a= |F_i|_a-1= |U_i|_a-1$. Thus $|U_i|_a >|a^{-1}F_i|_a$.
Since $a^{-1}F_i \in \Ff s$, one has
$$r_s(U_i)= r_t(V_i)= a .$$
Let us now consider the case $r_t(V_i)= b$, i.e., $|V_i|_a= |V'_i|_a-1$ and $|V_i|_b= |V'_i|_b+1$.
From (\ref{eq:uno}) and (\ref{eq:due}) one derives
$$|U_i|= |F_i|+1,$$
and $|U_i|_b>|F_i|_b$. Now $F_ia$ is a factor of $s$. Indeed, $F_i= L(V'_i)$ and for any $c\in \{a, b\}$
such that $V'_ic \in \Ff t$ one has $L(V'_ic)= F_iL(c)$. Since for any letter $c$, $L(c)$ begins with the letter $a$ it follows that $F_ia \in \Ff s$. Since $|F_ia| = |U_i|$ and $|U_i|_b >|F_i|_b = |F_ia|_b$, one has that $U_i$ is rich in $b$. Hence, $r_s(U_i)= r_t(V_i)=b$.

\vspace{2 mm}

\noindent
Case 2. All  $U_i$, $i\geq 1$, terminate with the letter $a$.

We consider the injective endomorphism  $R_a$ of $\{a, b\}^*$, simply denoted $R$,  defined in (\ref{morf:erre}).
 Since  $s$ is of type $a$ and $X=\{a, ba\}$ is a  prefix code, the word $s$ can be uniquely factored as a product of  elements of $X$. Thus there exists a unique word $t \in \{a,b\}^{\omega}$ such that  $s = R(t)$. The following holds:

\begin{itemize}
\item[1.] The word $t$ is a Sturmian word.
\item[2.] For any $i\geq 1$ there exists a non-empty prefix $V_i$ of $t$ such that $R(V_i)= U_i$.
\item[3.] The word $t$ can be factored as $ t= V_1\cdots V_n \cdots.$
\item[4.] $|V_1|<|U_1|$.
\item[5.] For all $i,j \geq 1$ one has $r_t(V_i)= r_t(V_j)$.
\end{itemize}

\noindent
Point 1. From Lemma \ref{lemma:onetwo}, since $R(t)=s$ it follows that $t$ is Sturmian.

\vspace{2 mm}

\noindent
Point 2.  For any $i\geq 1$, since $U_i$ terminates with the letter $a$  and any pair $(a, c)$ with $c\in \{a, b\}$ is synchronizing for $X^{\infty}$, one has  that $U_i\in X^*$, so that there exists $V_i\in \Pre t$ such that
$R(V_i)= U_i$.

\vspace{2 mm}

\noindent
Point 3. One has $R(V_1\cdots V_n\cdots) = U_1\cdots U_n \cdots= s = R(t)$. Thus
$t= V_1\cdots V_n \cdots$.

\vspace{2 mm}

\noindent
Point 4. By Lemma~\ref{lemma:cp}, $U_1$ is not a power of the first letter  $c$ of $s$, so that in $U_1$ there must be at least one occurrence of the letter $\bar c$. This implies that $|V_1|<|U_1|$.

\vspace{2 mm}

\noindent
Point 5. We shall prove that for all $i\geq 1$, $r_t(V_i)= r_s(U_i)$. From this  one has that $r_t(V_i)= r_t(V_j)$ for all $i,j \geq 1$.

Since $t$ is a Sturmian word, there exists $V'_i\in \Ff t$ such that
$$ |V_i|= |V'_i| \ \mbox{and} \ \mbox{either} \ |V_i|_a> |V'_i|_a \ \mbox{or} \  |V_i|_a< |V'_i|_a .$$
In the first case $r_t(V_i)= a$ and in the second case $r_t(V_i)= b$. Let us set
$$ F_i = R(V'_i).$$
Since $U_i = R(V_i)$, from the definition of the morphism $R$ one has that equations (\ref{eq:uno}) and
(\ref{eq:due}) are satisfied.

Let us first consider the case $r_t(V_i)= a$, i.e., $|V_i|_a= |V'_i|_a+1$ and $|V_i|_b= |V'_i|_b-1$.
From the preceding equations one has
$$ |F_i|= |U_i|+1.$$
From the definition of the morphism $R$ one has that $F_i= R(V'_i)$ terminates with the letter $a$.
Hence, $|F_ia^{-1}|= |U_i|$ and $|F_ia^{-1}|_a=  |F_i|_a-1= |U_i|_a-1$. Thus $|U_i|_a=|F_ia^{-1}|_a+1$, so that $U_i$ is rich in $a$ and $r_s(U_i)= r_t(V_i)=a$.

Let us now suppose that  $r_t(V_i)= b$, i.e., $|V_i|_a= |V'_i|_a-1$ and $|V_i|_b= |V'_i|_b+1$.
From (\ref{eq:uno}) and (\ref{eq:due}) one derives
$$|U_i|= |F_i|+1,$$
and $|U_i|_b>|F_i|_b$.
We prove that $aF_i \in \Ff s$. Indeed,  $F_i= R(V'_i)$ and for any $c\in \{a, b\}$
such that $cV'_i \in \Ff t$ one has $R(c)R(V'_i)= R(c)F_i$. Note that such a letter $c$  always  exists, as $t$ is recurrent. Since $R(c)$ terminates with the letter $a$ for every  letter $c$, it follows that $aF_i \in \Ff s$. Since $|aF_i| = |U_i|$ and $|U_i|_b >|aF_i|_b = |F_i|_b$, one has that $U_i$ is rich in $b$. Hence, $r_s(U_i)= r_t(V_i)=b$.
\end{proof}

\vspace{3 mm}

\noindent
({\em Proof of Theorem \ref{theorem:basic1}}) Let $s$ be a Sturmian word and suppose that $s$ admits a prefixal factorization
$s = U_1U_2 \cdots,$
where   $r_s(U_i)=r_s(U_j)$ for all $i,j \geq 1$.
Among all Sturmian words having this property we can always consider a Sturmian word $s$ such that
$|U_1|$ is minimal.  By Lemma~\ref{prop:basic},  there exists a Sturmian word $t$ such that $t$ admits a prefixal factorization
$t = V_1\cdots V_n \cdots,$
where  $r_t(V_i)=r_t(V_j)$ for all $i,j\geq 1$, and
$|V_1|<|U_1|$, which  contradicts the minimality of the length of $U_1$. $\Box$

\vspace{3 mm}

\noindent The following is an immediate consequence of Theorem~\ref{theorem:basic1}:

\begin{corollary}\label{cor:twenty-one} Let $s$ be a Sturmian word. Then $s\in {\cal P}_3$.
\end{corollary}
\begin{proof} Put $c(U)=r_s(U)$ if $U$ is a prefix of $s,$ and $c(U)=0$ otherwise.  It follows immediately from Theorem~\ref{theorem:basic1} that no factorization of $s$ is monochromatic.
\end{proof}

\noindent  Combining the preceding result with a theorem of  P. Hubert \cite{H},  we obtain

 \begin{corollary}Let $x$ be any aperiodic balanced word on any alphabet $A$. Then $x \in {\cal P}_3$.
 \end{corollary}
 \begin{proof}
 By the result of Hubert \cite[Theorem 1]{H}, there exist a partition $A=A'\cup A'' $ and a morphism $h: A^* \rightarrow \{a,b\}^*$  defined by $h(t)=a$ if $t\in A' $ and $h(t)=b$  if $t\in A''$, such that the word  $h(x)=y$ is Sturmian.
By Corollary \ref{cor:twenty-one}, we have $y \in {\cal P}_3$ and hence by Proposition \ref{invariance} and Remark~\ref{rem:lcolo}, we obtain $x \in {\cal P}_3$.
\end{proof}

\noindent Another application of Corollary \ref{cor:twenty-one} yields:

\begin{corollary}\label{complexity}Let $x\in A^\omega$ be a recurrent word and $N$ a positive integer. Suppose the factor complexity $\lambda_x(n)=n+k$, $k>0$, whenever $n\geq N.$ Then  $x \in {\cal P}$.
\end{corollary}

\begin{proof} We have $\lambda_x(n+1)-\lambda_x(n)=1$ for $n\geq N$, whence $x$ has exactly one right special factor of each length $n\geq N$ and this right special factor has exactly two right extensions to a factor of length $n+1.$ Since $x$ is aperiodic,  any non-empty factor $v$ of $x$  has more than one  return word, i.e.,   $\card(\mathcal{R}_v(x))>1$.  We claim that  $\card(\mathcal{R}_v(x))=2$ for each factor $v$ of $x$ with $|v|\geq N.$ It suffices to prove the claim for each right special factor $v$ of $x$ with $|v|\geq N.$ Suppose to the contrary that some right special factor $v$ of $x$ with $|v|\geq N$ has three complete  return words  $z_1,z_2,z_3.$ For each $1\leq i< j\leq 3$
let $w_{i,j}$ denote the longest common prefix of $z_i$ and $z_j.$ Then $|w_{i,j}|\geq |v|$ for each $1\leq i< j\leq 3$ and for some choice of $i$ and $j$ we have $|w_{i,j}|>|v|.$ Since each $w_{i,j}$ is a right special factor of $x$,  it follows that $v$ is both a proper prefix and a proper suffix of some $w_{i,j},$ contradicting that each $z_i$ contains exactly two occurrences of $v.$
Now fix a prefix $u$ of $x$ with $|u|\geq N.$ It follows from the previous claim that
$\card(\mathcal{R}_v(\mathcal{D}_u(x)))=2$ for every factor $v$ of $\mathcal{D}_u(x).$ Via a result of  L. Vuillon  \cite{Vu}  we conclude that $\mathcal{D}_u(x)$ is a Sturmian word. The result now follows from Corollary~\ref{cor:twenty-one} and Corollary~\ref{cor: cprw}.
\end{proof}

An infinite word $x$ is called {\em episturmian} \cite{DJP}  if there exists a standard episturmian word (see Remark~\ref{rem:eli}) $y$
such that $\Ff  x = \Ff  y$. We have seen in Section \ref{bpcp} (cf. Remark~\ref{rem:eli})  that the non-periodic standard epi\-stur\-mian words are in ${\cal P}$. A problem that naturally arises is whether, as in the case of  Sturmian words,  all non-periodic episturmian words are in  ${\cal P}$. We end with a partial result in this direction:

\begin{proposition}Let $x\in A^\omega$ be a suffix of a non-periodic standard episturmian word $y.$ Then  $x \in {\cal P}$.
\end{proposition}

\begin{proof} Let us write $y=ux$ with $u\in A^*.$ We first claim that there exists a positive integer $N>|u|$ such that for every prefix $v$ of $x$ with $|v|\geq N,$ every occurrence of $v$ in $x$ is immediately preceded by $u$, i.e., if $x= \lambda v \xi$ with $\lambda \in A^*, \xi\in A^{\omega}$, and $|v|\geq N$, then $\lambda= \lambda' u$ with $\lambda'\in A^*$.  This is clear  when  $u$ is the empty word. So suppose $|u|=n\geq 1,$  and suppose to the contrary that there exist infinitely many prefixes of $x$ occurring
in $x$ immediately preceded by a word of length $n$ different from $u$. By  the pigeonhole principle,  there exists $u'\in A^n$ different from $u$ and  infinitely many prefixes $v$ of $x$ occurring in $x$ preceded by $u'.$ It follows that $\Ff u'x \subseteq \Ff y$. Let $s\in A^*$ denote the longest common suffix of $u$ and $u'$,  so $u=rs$ and $u'=r's$, where $r$ and $r'$ terminate with different letters. It follows that every prefix of $sx$ is a left special factor of $y.$ As $y$ is a non-periodic standard episturmian word, every prefix of $y$ is a left special factor of $y.$ Since every episturmian word has at most one left special factor of each length, it follows that $y=sx,$ contradicting that $y$ is not periodic. This establishes the claim.

Now we color the factors $z$ of $x$ by $|z|$ if $z$ is a prefix of $x$ with $|z|\leq N,$ by the suffix of $z$ of length $|u|+1$ if $z$ is a prefix of $x$ with $|z|>N,$ and by $A$ if $z$ is not a prefix of $x.$
Then, since $x$ is not periodic, a monochromatic factorization of $x$ corresponds to a prefixal factorization $x=U_1U_2U_3\cdots$  in which each prefix $U_i$ of $x$ terminates with the same suffix of length $|u|+1,$  (call it $v).$ By the previous claim, we can write $v=au$ for some $a\in A.$ But this defines a prefixal factorization of $ay=aux$: namely $ay=(auU_1v^{-1})(vU_2v^{-1})(vU_3v^{-1})\cdots$. Since $y$ is a non-periodic standard episturmian word, by Corollary~\ref{cor:N1} (see also Remark~\ref{rem:eli}), $ay$ does not admit a prefixal factorization,  a contradiction.
\end{proof}

\vspace{2 mm}

 \noindent
{\bf Acknowledgments:} We  are deeply indebted to Tom Brown for his suggestions and comments.  We thank three anonymous reviewers for many useful comments.

\end{document}